\documentclass{article}

\usepackage{indentfirst}
\usepackage{amsmath,amsfonts,amsthm,amssymb}
\usepackage{mathrsfs}
\usepackage{amscd}
\usepackage{hyperref}
\usepackage[all]{xy}

\def\co{\colon\thinspace}
\DeclareMathAlphabet{\mathsfsl}{OT1}{cmss}{m}{sl}

\newtheorem{thm}{Theorem}[section]
\newtheorem{lem}[thm]{Lemma}
\newtheorem{cor}[thm]{Corollary}
\newtheorem{prop}[thm]{Proposition}
\newtheorem{conj}[thm]{Conjecture}

\theoremstyle{definition}
\newtheorem{defn}[thm]{Definition}

\newtheorem{rem}[thm]{Remark}

\newcommand{\on}{\operatorname}

\newcommand{\Spinc}{\on{Spin}^c}

\newcommand{\Z}{\mathbb{Z}}

\newcommand\goth[1]{\mathfrak{#1}}
\newcommand{\s}{\goth{s}}

\begin{document}

\title{Cosmetic surgeries on knots in $S^3$}

\author{{\large Yi NI} and {\large Zhongtao WU}\\{\normalsize Department of Mathematics, Caltech, MC 253-37}\\
{\normalsize 1200 E California Blvd, Pasadena, CA
91125}\\{\small\it Emai\/l\/:\quad\rm yini@caltech.edu\quad zhongtao@caltech.edu}}

\date{}
\maketitle

\begin{abstract}
Two Dehn surgeries on a knot are called {\it purely cosmetic}, if they yield manifolds that are homeomorphic as oriented manifolds. Suppose there exist purely cosmetic surgeries
on a knot in $S^3$,
we show that the two surgery slopes must be the opposite of each other. One ingredient of our proof is a Dehn surgery formula for correction terms in Heegaard Floer homology.
\end{abstract}

\section{Introduction}

Given a knot $K$ in a three-manifold $Y$, let $\alpha,\beta$ be two different slopes on $K$, and let $Y_{\alpha}(K),Y_{\beta}(K)$ be the manifolds obtained by $\alpha$-- and $\beta$--surgeries on $K$, respectively.
If $Y_{\alpha}(K),Y_{\beta}(K)$ are homeomorphic, then we say the two surgeries are {\it cosmetic}; if $Y_{\alpha}(K)\cong Y_{\beta}(K)$ as oriented manifolds, then these two surgeries are {\it purely cosmetic};
if $Y_{\alpha}(K)\cong -Y_{\beta}(K)$ as oriented manifolds, then these two surgeries are {\it chirally cosmetic}.

Chirally cosmetic surgeries occur frequently for knots in $S^3$. For example, if $K$ is amphicheiral, then $S^3_{r}(K)\cong-S^3_{-r}(K)$ for any slope $r$. If $T$ is the right hand trefoil knot, then
$S^3_{(18k+9)/(3k+1)}(T)\cong-S^3_{(18k+9)/(3k+2)}(T)$
for any nonnegative integer $k$ \cite{Mathieu}.

On the other hand, purely cosmetic surgeries are very rare.  In fact, the following conjecture was proposed in Gordon's ICM talk \cite[Conjecture~6.1]{GordonICM} and Kirby's Problem List \cite[Problem~1.81 A]{Kirby}.

\begin{conj}[Cosmetic Surgery Conjecture]\label{conj:Cosmetic}
Suppose $K$ is a knot in a closed oriented three-manifold $Y$ such that $Y-K$ is irreducible and not homeomorphic to the solid torus. If two different Dehn surgeries on $K$ are purely cosmetic, then there is a homeomorphism of $Y-K$ which takes one slope to the other.
\end{conj}

This conjecture is highly nontrivial even when $Y=S^3$. In
\cite{GL}, Gordon and Luecke proved the famous Knot Complement
Theorem, which can be interpreted as that there are no cosmetic
surgeries if one of the two slopes is $\infty$. In \cite{BL},
Boyer and Lines proved the cosmetic surgery conjecture for any
knot $K$ with $\Delta ''_K(1)\neq 0$. In recent years, Heegaard
Floer homology \cite{OSzAnn1} became a powerful tool to study this
conjecture. In \cite{OSzRatSurg}, Ozsv\'ath and Szab\'o proved
that if $S^3_{r_1}(K)$ is homeomorphic to $S^3_{r_2}(K)$, then
either $S^3_{r_1}(K)$ is an $L$--space or $r_1$ and $r_2$ have
opposite signs. Moreover, when $S^3_{r_1}(K)$ is homeomorphic to
$S^3_{r_2}(K)$ as oriented manifolds, Wu \cite{Wu} ruled out the
case that $S^3_{r_1}(K)$ is an $L$--space, thus $r_1$ and $r_2$
must have opposite signs. In \cite{Wang}, Wang proved that genus
$1$ knots in $S^3$ do not admit purely cosmetic surgeries.

In this paper, we are going to put more restrictions on purely cosmetic surgeries for knots in $S^3$. Our main result is:

\begin{thm}\label{thm:OpSlope}
Suppose $K$ is a nontrivial knot in $S^3$, and $r_1,r_2\in\mathbb
Q\cup\{\infty\}$ are two distinct slopes such that $S^3_{r_1}(K)$
is homeomorphic to $S^3_{r_2}(K)$ as oriented manifolds. Then
$r_1,r_2$ satisfy that

(a) $r_1=-r_2$;

(b) if $r_1=p/q$, where $p,q$ are coprime integers, then
$$q^2\equiv-1\pmod p;$$

\noindent and $K$ satisfies

(c) $\tau(K)=0$, where $\tau$ is the concordance invariant defined
by Ozsv\'ath--Szab\'o \cite{OSz4Genus} and Rasmussen
\cite{RasThesis}.
\end{thm}

\begin{rem}
Suppose $K\subset S^3$ is a nontrivial knot, and $f$ is a
homeomorphism of  $S^3-K$. Then $f$ must send the longitude to the
longitude (up to orientation reversing) for homological reason, and the
meridian to the meridian (up to orientation reversing) by the Knot
Complement Theorem \cite{GL}. Thus if $f$ sends a slope to a
different slope, then $f$ extends to an orientation reversing
homeomorphism of $S^3$, which means that $K$ is amphicheiral. So
Conjecture~\ref{conj:Cosmetic} for knots in $S^3$ can be
reformulated as: if $S^3_{r_1}(K)$ is homeomorphic to
$S^3_{r_2}(K)$ as oriented manifolds, then $r_1=-r_2$ and $K$ is
amphicheiral. Our Theorem~\ref{thm:OpSlope} proves the part of
this conjecture concerning the slopes, and asserts that $\tau(K)$
is the same as the $\tau$ invariant of amphicheiral knots.
\end{rem}

\begin{rem}
Ozsv\'ath and Szab\'o \cite{OSzRatSurg} remarked that their method can be used to exclude cosmetic surgeries for certain numerators $p$.
To illustrate, they proved that $p\ne3$. Our Theorem~\ref{thm:OpSlope} (b) implies a more precise restriction on $p$: $-1$ must be a quadratic residue modulo $p$.
\end{rem}

\begin{rem}
Ozsv\'ath and Szab\'o \cite{OSzRatSurg} gave the example of $K=9_{44}$. This knot is a genus $2$ knot with $\tau(K)=0$ and
$$\Delta_K(T)=T^{-2}-4T^{-1}+7-4T+T^2.$$
Heegaard Floer homology does not obstruct $K$ from admitting purely cosmetic surgeries. In fact, $S^3_{1}(K)$ and $S^3_{-1}(K)$ have the same Heegaard Floer homology.
However, these two manifolds are not homeomorphic since they have different hyperbolic volumes.
\end{rem}

An important ingredient of the proof of Theorem~\ref{thm:OpSlope} is a surgery formula for Ozsv\'ath and Szab\'o's correction terms. The statement of the formula is as follows, where the terms in the formula will be explained in Section~2.

\begin{prop}\label{prop:Corr}
Suppose $p,q>0$, and fix $0\le i\le p-1$. Then
$$d(S^3_{p/q}(K),i)=d(L(p,q),i)-2\max\{V_{\lfloor\frac{i}q\rfloor},H_{\lfloor\frac{i+p(-1)}q\rfloor}\}.$$
\end{prop}

The above formula has independent interests, since the correction terms have been very useful in many applications of Heegaard Floer homology. 

We also compute the rank of the reduced Heegaard Floer homology of surgeries on knots which admits purely cosmetic surgeries. 
 
\begin{prop}\label{prop:Const}
Suppose $K\subset S^3$ is a knot with $S^3_{r}(K)\cong S^3_{-r}(K)$ for some $r\in\mathbb Q\setminus\{0\}$. Then there exists a constant $C=C(K)$, such that
$$\mathrm{rank}\:HF_{\mathrm{red}}(S^3_{p/q}(K))=|q|\cdot C,$$
for any coprime nonzero integers $p,q$. In fact, the constant $C(K)$ is the
rank of $HF_{\mathrm{red}}(S^3_{p}(K))$ for $p>0$.
\end{prop}

This paper is organized as follows. In Section~\ref{sect:RatForm}, we use Ozsv\'ath and Szab\'o's rational surgery formula \cite{OSzRatSurg} to prove Proposition~\ref{prop:Corr}. This gives a bound of the correction terms by the correction terms of the corresponding lens spaces. A necessary and sufficient condition for the bound to be reached is found.
In Section~\ref{sect:CGCW}, we review the Casson--Walker and Casson--Gordon invariants. Combining these with the bound obtained in Section~\ref{sect:RatForm}, we show that if there are purely cosmetic surgeries,
then the correction terms are exactly the correction terms of the corresponding lens spaces. Our main theorem is then proved in Section~\ref{sect:Main}.
In Section~\ref{sect:Red}, we will prove Proposition~\ref{prop:Const}.

\vspace{5pt}\noindent{\bf Acknowledgements.} This work was carried out when the first author visited Princeton University.
The first author was partially supported by an AIM Five-Year Fellowship and NSF grant
number DMS-1021956. The second author was supported by a Simons Postdoctoral Fellowship.


\section{Rational surgeries and the correction terms}\label{sect:RatForm}

\subsection{The rational surgery formula}

In this subsection, we recall the rational surgery formula of Ozsv\'ath and Szab\'o \cite{OSzRatSurg}, and then compute the example of surgeries on the unknot.

\begin{rem}
For simplicity, throughout this paper we will use $\mathbb F_2=\mathbb Z/2\mathbb Z$ coefficients for Heegaard Floer homology. Our proofs work for $\mathbb Z$ coefficients as well.
\end{rem}

Given a knot $K$ in an integer homology sphere $Y$. Let $C=CFK^{\infty}(Y,K)$ be the knot Floer chain complex of $(Y,K)$. There are chain complexes
$$A^+_k=C\{i\ge0 \text{ or }j\ge k\},\quad k\in\mathbb Z$$
and $B^+=C\{i\ge0\}\cong CF^+(Y)$. As in \cite{OSzIntSurg}, there are chain maps
$$v_k,h_k\co A^+_k\to B^+.$$

Given $\frac pq\in\mathbb Q\setminus\{0\}$,
let $$\mathbb A_i^+=\bigoplus_{s\in\mathbb
Z}(s,A^+_{\lfloor\frac{i+ps}q\rfloor}(K)),\mathbb
B_{i}^+=\bigoplus_{s\in\mathbb Z}(s,B^+).$$ Here,the first entry
$s$ in the parentheses is simply a decoration used to distinguish
identical copies of $A^+_k$ or $B^+$.
Define maps
$$v^+_{\lfloor\frac{i+ps}q\rfloor}\co (s,A^+_{\lfloor\frac{i+ps}q\rfloor}(K))\to (s,B^+),\quad h^+_{\lfloor\frac{i+ps}q\rfloor}\co (s,A^+_{\lfloor\frac{i+ps}q\rfloor}(K))\to (s+1,B^+).$$
Adding these up, we get a chain map
$$D_{i,p/q}^+\co\mathbb A_i^+\to \mathbb B_i^+,$$
with
$$D_{i,p/q}^+\{(s,a_s)\}_{s\in\mathbb Z}=\{(s,b_s)\}_{s\in\mathbb Z},$$
where
$$b_s=v^+_{\lfloor\frac{i+ps}q\rfloor}(a_s)+h^+_{\lfloor\frac{i+p(s-1)}q\rfloor}(a_{s-1}).$$

In \cite{OSzRatSurg}, there is an identification of $\mathrm{Spin}^c(Y_{p/q}(K))$ with $\mathbb Z/p\mathbb Z$. This identification can be made explicit by the procedure in \cite[Section~4, Section~7]{OSzRatSurg}. For our purpose in this paper, we only need to know that such an identification exists. We will use this identification throughout this paper.

\begin{thm}[Ozsv\'ath--Szab\'o]\label{thm:MC}
Let $\mathbb X^+_{i,p/q}$ be the mapping cone of $D_{i,p/q}^+$, then there is a relatively graded isomorphism of groups
$$H_*(\mathbb X^+_{i,p/q})\cong HF^+(Y_{p/q}(K),i).$$
\end{thm}

The above isomorphism is actually $U$--equivariant, so the two groups are isomorphic as $\mathbb F_2[U]$--modules.

\begin{rem}\label{rem:AbGr}
For $K\subset S^3$, the absolute grading on $\mathbb X_{i,p/q}^+$ is determined by an absolute grading on $\mathbb B^+$ which is independent of $K$.
The absolute grading on $\mathbb B^+$ is chosen so that the grading of $$\mathbf1\in H_*(\mathbb X^+_{i,p/q}(O))\cong\mathcal T^+:=\mathbb F_2[U,U^{-1}]/U\mathbb F_2[U]$$ is $d(L(p,q),i)$, where $O$ is the unknot.
This absolute grading on $\mathbb X^+_{i,p/q}(K)$ then agrees with the absolute grading on $HF^+(S^3_{p/q}(K),i)$. See \cite[Section~7.2]{OSzRatSurg} for a discussion of the absolute grading.
\end{rem}

Let $$\mathfrak A_k^+=H_*(A_k^+), \, \mathfrak B^+=H_*(B^+),$$
let $$\mathfrak v_k^+, \mathfrak h_k^+\co \mathfrak A_k^+\to \mathfrak B^+$$
be the maps induced on homology,
and let $$\mathfrak D_{i,p/q}^+\co H_*(\mathbb A_i^+)\to H_*(\mathbb
B_i^+)$$ be the map induced by $D_{i,p/q}^+$ on homology. 
There is a natural short exact sequence of chain complexes:
$$
\xymatrix{
0\ar[r]&\mathbb B^+ \ar[r]^-{\mathrm{incl}} & \mathbb X_{i,p/q}^+\ar[r]^-{\mathrm{proj}} &\mathbb A^+_i\ar[r] &0.
}
$$
Then Theorem~\ref{thm:MC} implies that there is an exact triangle
\begin{equation}\label{eq:TriangleMC}
\xymatrix{
H_*(\mathbb A_{i}^+) \ar[r]^-{\mathfrak D_{i,p/q}^+} &H_*(\mathbb B^+)\ar[d]^-{\mathrm{incl}_*}\\
&HF^+(Y_{p/q}(K),i).\ar[lu]^-{\mathrm{proj}_*}
}
\end{equation}


If $K=O$ is the unknot, then the $\frac pq$--surgery on $K$ gives rise to the lens space $L(p,q)$.
Then $$\mathfrak A_k^+\cong A^+_k\cong\mathfrak B_k^+\cong  B^+_k\cong\mathcal T^+.$$ We have
$$v_k^+=\left\{
\begin{array}{ll}
U^{|k|}&\text{if } k\le0,\\
1&\text{if }k\ge0,
\end{array}
\right.$$
$$h_k^+=\left\{
\begin{array}{ll}
1&\text{if } k\le0,\\
U^{|k|}&\text{if }k\ge0.
\end{array}
\right.$$

Suppose $p,q>0$. Let $0\le i\le p-1$, then
$\lfloor\frac{i+ps}q\rfloor\ge0$ if and only if $s\ge0$.
We have $b_{0}=a_{0}+a_{-1}$.
For $\xi\in\mathcal T^+$, define $$\iota(\xi)=\{(s,\xi_s)\}_{s\in\mathbb Z}\in\mathbb A^+_i$$
by letting $$
\left\{
\begin{array}{ll}
\xi_{0}=\xi_{-1}=\xi,&\\
\xi_s=U^{\lfloor\frac{i+p(s-1)}q\rfloor}\xi_{s-1},&\text{if }s>0,\\
\xi_s=U^{-\lfloor\frac{i+p(s+1)}q\rfloor}\xi_{s+1},&\text{if }s<-1.
\end{array}
\right.
$$

Then $\iota$ maps $\mathcal T^+$ isomorphically to the kernel of
$D^+_{i,p/q}$. It is clear that $D^+_{i,p/q}$ is surjective, (see
also Lemma~\ref{lem:>0Surj} for a proof,) hence $\ker D^+_{i,p/q}
\cong \mathbb X^+_{i,p/q}$. In particular, $\iota(\mathbf1)$
should have absolute grading $d(L(p,q),i)$. The absolute grading
on $\mathbb B^+$ can be determined by the fact
\begin{equation}\label{eq:GrL(p,q)}
v^+_{\lfloor\frac{i}q\rfloor}(0,\mathbf1)=h^+_{\lfloor\frac{i+p(-1)}q\rfloor}(-1,\mathbf1)=(0,\mathbf1)\in(0,B^+).
\end{equation}

\subsection{Bounding the correction terms}

For a rational homology three-sphere $Y$ equipped with a $\Spinc$ structure $\s$, $HF^+(Y,\s)$ can be decomposed as the direct sum of two groups: The first group is the image of $HF^\infty(Y,\s)$ in $HF^+(Y,\s)$, whose minimal absolute $\mathbb{Q}$ grading is an invariant of $(Y,\s)$ and is denoted by $d(Y,\s)$, the {\it correction term} \cite{OSzAbGr}; the second group is the quotient modulo the above image and is denoted by $HF_{\mathrm{red}}(Y,\s)$.  Altogether, we have $$HF^+(Y,\s)=\mathcal{T}^+_{d(Y,\s)}\oplus HF_{\mathrm{red}}(Y,\s).$$

For a knot $K\subset S^3$, let $\mathfrak A_k^T=U^n\mathfrak A^+_k$ for $n\gg0$, then
$\mathfrak A_k^T\cong\mathcal T^+$. Let $\mathfrak D_{i,p/q}^T$ be the restriction of
$\mathfrak D_{i,p/q}^+$ on
$$\mathcal A_i^T=\bigoplus_{s\in\mathbb Z}(s,\mathfrak A^T_{\lfloor\frac{i+ps}q\rfloor}(K)).$$

Since $\mathfrak v^+_k,\mathfrak h^+_k$ are graded isomorphisms at sufficiently high
gradings and are $U$--equivariant, $\mathfrak v^+_k|\mathfrak A_k^T$ is modeled on multiplication by
$U^{V_k}$ and $\mathfrak h^+_k|\mathfrak A_k^T$ is modeled on multiplication by
$U^{H_k}$, where $V_k,H_k\ge0$. Note that the numbers $V_k$ and
$H_k$ are invariants of the double-filtered chain homotopy
equivalence type of $CFK^\infty(Y,K)$.  Hence, they are invariants
of the knot $K$.

\begin{lem}\label{lem:VHMono}
$V_k\ge V_{k+1}$, $H_k\le H_{k+1}$.
\end{lem}
\begin{proof}
The map $v^+_k$ factors through the map $v^+_{k+1}$ via the factorization
$$
\begin{CD}
C\{i\ge0\text{ or }j\ge k\}@>>> C\{i\ge0\text{ or }j\ge k+1\}@>v^+_{k+1}>>C\{i\ge0\}.
\end{CD}
$$
Hence it is easy to see that $V_k\ge V_{k+1}$. A similar argument works for $H_k$ by considering the factorization of $h^+_{k+1}$:
$$
\begin{CD}
C\{i\ge0\text{ or }j\ge k+1\}@.@.\\
@VUVV \\
C\{i\ge-1\text{ or }j\ge k\}
@>>> C\{i\ge0\text{ or }j\ge k\}@>h^+_{k}>>C\{i\ge0\}.
\end{CD}
$$
\end{proof}

It is obvious that $V_k=0$ when $k\ge g$ and $H_k=0$ when $k\le -g$, $V_k\to +\infty$ as $k\to -\infty$, $H_k\to +\infty$ as $k\to+\infty$.

\begin{thm}\label{thm:Grading}
Suppose $p,q>0$ are coprime integers. Then
$$d(S^3_{p/q}(K),i)\le d(L(p,q),i)$$
for all $i\in\mathbb Z/p\mathbb Z$. The equality holds for all $i$ if and only if $V_0=H_0=0$.
\end{thm}

\begin{rem}
The first part of Theorem~\ref{thm:Grading} easily follows from \cite[Theorem~9.6]{OSzAbGr} and \cite[Corollary~1.5]{OSzPlumb}.
We will present a different proof, which enables us to get the conclusion about $V_0$ and $H_0$.
\end{rem}

\begin{lem}\label{lem:V0H0}
For any knot $K\subset S^3$, we have $V_0=H_0$. Hence $V_k\ge H_k$
if $k\le0$ and $V_k\le H_k$ if $k\ge0$.
\end{lem}
\begin{proof}
If $$(\Sigma,\mbox{\boldmath${\alpha}$},
\mbox{\boldmath$\beta$},w,z)$$ is a doubly pointed Heegaard
diagram for $(S^3,K)$, then
$$(-\Sigma,\mbox{\boldmath${\beta}$},
\mbox{\boldmath$\alpha$},z,w)$$
is also a Heegaard diagram for $(S^3,K)$. Hence the roles of $i,j$ can be interchanged. It follows that $v^+_0$ is equivalent to $h^+_0$, hence $V_0=H_0$.
\end{proof}

\begin{lem}\label{lem:>0Surj}
Suppose $p/q>0$. Then the map $\mathfrak D_{i,p/q}^T$ is surjective for each
$0\le i \le p-1$.
\end{lem}
\begin{proof}
Suppose
$$\mbox{\boldmath${\eta}$}=\{(s,\eta_s)\}_{s\in\mathbb Z}\in H_*(\mathbb B^+_i).$$
Let
$$\xi_{-1}=U^{-H_{\lfloor\frac{i+p(-1)}q\rfloor}}\eta_{0},\quad \xi_{0}=0.$$
Here $\xi_{-1}=U^{-H_{\lfloor\frac{i+p(-1)}q\rfloor}}\eta_{0}$ means that $\xi_{-1}$ is an element with $U^{H_{\lfloor\frac{i+p(-1)}q\rfloor}}\xi_{-1}=\eta_{0}$.
For other $s$, let
$$
\xi_{s}=\bigg\{
\begin{array}{ll}
U^{-V_{\lfloor\frac{i+ps}q\rfloor}}(\eta_{s}-U^{H_{\lfloor\frac{i+p(s-1)}q\rfloor}}\xi_{s-1}),&\text{if }s>0,\\
U^{-H_{\lfloor\frac{i+ps}q\rfloor}}(\eta_{s+1}-U^{V_{\lfloor\frac{i+p(s+1)}q\rfloor}}\xi_{s+1}),&\text{if }s<-1.
\end{array}
$$
By the definition of direct sum, $\eta_s=0$ when $|s|\gg0$. Using the facts that
$$H_{\lfloor\frac{i+p(s-1)}q\rfloor}-V_{\lfloor\frac{i+ps}q\rfloor}\to+\infty, \text{ as }s\to+\infty,$$
and
$$V_{\lfloor\frac{i+p(s+1)}q\rfloor}-H_{\lfloor\frac{i+ps}q\rfloor}\to+\infty, \text{ as }s\to-\infty,$$
we see that $\xi_s=0$ when $|s|\gg0$. So $\mbox{\boldmath${\xi}$}=\{(s,\xi_s)\}_{s\in\mathbb Z}\in \mathcal A_i^T$.
Clearly
$$\mathfrak D_{i,p/q}^T(\mbox{\boldmath${\xi}$})=\mbox{\boldmath${\eta}$}.$$
\end{proof}

Our key idea is the following lemma.

\begin{lem}\label{lem:Tpart}
Suppose $p/q>0$. Under the identification
$$H_*(\mathbb{X}^+_{i,p/q})\cong HF^+(S^3_{p/q}(K),i),$$
$U^n HF^+(S^3_{p/q}(K),i)$ is identified with a subgroup of the homology of the mapping cone of $D_{i,p/q}^T$ when $n\gg0$.
\end{lem}
\begin{proof}
By Lemma~\ref{lem:>0Surj}, $\mathfrak D_{i,p/q}^T$ is surjective, hence $\mathfrak D_{i,p/q}^+$ is also surjective. By the exact triangle (\ref{eq:TriangleMC}),
we conclude that $HF^+(Y_{p/q}(K),i)\cong \ker\mathfrak D_{i,p/q}^+$.
Suppose $\mbox{\boldmath${\xi}$}\in U^n\ker\mathfrak D_{i,p/q}^+$ for $n\gg0$, then $\mbox{\boldmath${\xi}$}\in U^n H_*(\mathbb A_{i,p/q}^+)=\mathcal A_{i,p/q}^T$.
Hence $\mbox{\boldmath${\xi}$}$, being an element in $\ker\mathfrak D_{i,p/q}^+$, is actually an element in $\ker\mathfrak D_{i,p/q}^T$. This proves that
$U^n\ker\mathfrak D_{i,p/q}^+$ is a subgroup of the homology of the mapping cone of $D_{i,p/q}^T$.
\end{proof}

Now we are ready to prove Proposition~\ref{prop:Corr} stated in the introduction.

\begin{proof}[Proof of Proposition~\ref{prop:Corr}]
Lemmas~\ref{lem:VHMono} and~\ref{lem:V0H0} imply that
\begin{equation}\label{eq:si}
\begin{array}{ll}
H_{\lfloor \frac{i+p(s-1)}q\rfloor}\ge H_0=V_0\ge V_{\lfloor \frac{i+ps}q\rfloor} &\text{if }s>0,\\
H_{\lfloor \frac{i+p(s-1)}q\rfloor}\le H_0=V_0\le V_{\lfloor \frac{i+ps}q\rfloor} &\text{if }s<0.
\end{array}
\end{equation}

Given $\xi\in\mathcal T^+$, define
$$\rho(\xi)=\{(s,\xi_s)\}_{s\in\mathbb Z}$$
as follows. If
\begin{equation}\label{eq:V>H}
V_{\lfloor\frac{i}q\rfloor}\ge H_{\lfloor\frac{i+p(-1)}q\rfloor},
\end{equation}
let
$$\xi_{-1}=U^{V_{\lfloor\frac{i}q\rfloor}-H_{\lfloor\frac{i+p(-1)}q\rfloor}}\xi,\quad \xi_{0}=\xi;$$
if
\begin{equation}\label{eq:V<H}
V_{\lfloor\frac{i}q\rfloor}< H_{\lfloor\frac{i+p(-1)}q\rfloor},
\end{equation}
let
$$\xi_{-1}=\xi,\quad \xi_{0}=U^{H_{\lfloor\frac{i+p(-1)}q\rfloor}-V_{\lfloor\frac{i}q\rfloor}}\xi.$$
For other $s$, using (\ref{eq:si}), let
$$
\xi_s=\bigg\{
\begin{array}{ll}
U^{H_{\lfloor\frac{i+p(s-1)}q\rfloor}-V_{\lfloor\frac{i+ps}q\rfloor}}\xi_{s-1},&\text{if }s>0,\\
U^{V_{\lfloor\frac{i+p(s+1)}q\rfloor}-H_{\lfloor\frac{i+ps}q\rfloor}}\xi_{s+1},&\text{if }s<-1.
\end{array}
$$
As argued in the proof of Lemma~\ref{lem:>0Surj}, $\xi_s=0$ when $|s|\gg0$.
Then $\rho$ maps $\mathcal T^+$ into $\ker\mathfrak D^T_{i,p/q}$ and $U\rho(\mathbf 1)=0$. In light of Lemma~\ref{lem:Tpart}, the grading of $\rho(\mathbf1)$ is $d(S^3_{p/q}(K),i)$.

If (\ref{eq:V>H}) holds, the map $$\mathfrak v^+_{\lfloor \frac{i}q\rfloor}\co(0,\mathfrak A^T_{\lfloor \frac{i}q\rfloor})\to(0,\mathfrak B^+)$$ is
$U^{V_{\lfloor \frac{i}q\rfloor}}$.
Using Remark~\ref{rem:AbGr} and comparing (\ref{eq:GrL(p,q)}), the grading of $\rho(\mathbf1)$ can be computed by
$$d(L(p,q),i)-2V_{\lfloor \frac{i}q\rfloor}.$$
If (\ref{eq:V<H}) holds,
the map $$\mathfrak h^+_{\lfloor \frac{i+p(-1)}q\rfloor}\co(-1,\mathfrak A^T_{\lfloor \frac{i+p(-1)}q\rfloor})\to(0,\mathfrak B^+)$$ is
$U^{H_{\lfloor \frac{i+p(-1)}q\rfloor}}$.
The grading of $\rho(\mathbf1)$ can be computed by
$$d(L(p,q),i)-2H_{\lfloor \frac{i+p(-1)}q\rfloor}.$$
\end{proof}

\begin{rem}
The argument in the proof of Lemma~\ref{lem:V0H0} implies that $V_k=H_{-k}$ for any $k\in\mathbb Z$. So Proposition~\ref{prop:Corr} may be stated as
$$d(S^3_{p/q}(K),i)=d(L(p,q),i)-2\max\{V_{\lfloor\frac{i}q\rfloor},V_{\lfloor\frac{p+q-1-i}q\rfloor}\}.$$
\end{rem}

\begin{proof}[Proof of Theorem~\ref{thm:Grading}]
The first part of Theorem~\ref{thm:Grading} immediately follows from Proposition~\ref{prop:Corr}.

If $d(S^3_{p/q}(K),i)=d(L(p,q),i)$ for all $i$, then
\begin{equation}\label{eq:MaxVH=0}
\max\{V_{\lfloor\frac{i}q\rfloor},H_{\lfloor\frac{i+p(-1)}q\rfloor}\}=0
\end{equation}
for all $i$. In particular, $V_{0}=0$. It follows from Lemma~\ref{lem:V0H0} that $H_0=0$.

If $V_0=H_0=0$, then (\ref{eq:MaxVH=0}) holds for all $i$. So $d(S^3_{p/q}(K),i)=d(L(p,q),i)$.
\end{proof}


\section{Casson--Walker, Casson--Gordon  invariants and the correction term}\label{sect:CGCW}

\subsection{Casson--Walker invariant}

The Casson invariant is one of the many invariants of a closed three-manifold $Y$ that can be obtained by studying representations of its fundamental group in a certain non-abelian group $G$.
Roughly speaking, the Casson invariant of an integral homology sphere $Y$ is obtained by counting representations of $\pi_1(Y)$ in $G=SU(2)$.  The geometric structure used to obtain a topological invariant is a Heegaard splitting of $Y$.  An alternative gauge-theoretical approach uses flat bundles together with a Riemannian metric on $Y$ and leads to a refinement of the Casson invariant, the Floer homology.

By extending Casson's $SU(2)$ intersection theory to include
reducible representations, Walker extended the Casson invariant to
rational homology spheres.  Most remarkably, Walker's invariant
admits a purely combinatorial definition in terms of surgery
presentations.  The following proposition is the special case of a
more general surgery formula, when $K$ is a null-homologous knot
in a rational homology sphere $Y$ (see \cite[Theorem 2.8]{BL}).
Our convention here is that $\lambda(S^3_{+1}(T))=1$, where $T$ is
the right hand trefoil.

\begin{prop}
 Let $K$ be a null-homologous knot in a rational homology three-sphere $Y$, and let $L(p,q)$ be the lens space obtained by $(p/q)$--surgery on the unknot in $S^3$.  Then
\begin{equation}\label{eq:surgery1}
\lambda(Y_{p/q}(K))=\lambda(Y)+\lambda(L(p,q))+\frac{q}{2p}\Delta_K''(1).
\end{equation}
Here, the Alexander polynomial $\Delta_K$ is normalized to be symmetric and
satisfy $\Delta_K(1)=1$.
\end{prop}


\begin{defn}
Given two coprime numbers $p$ and $q$, the {\it Dedekind sum} $s(q,p)$ is $$s(q,p):=\text{sign}(p)\cdot \sum^{|p|-1}_{k=1}((\frac{k}{p}))((\frac{kq}{p})),$$
where $$ ((x))=\begin{cases} x-\lfloor x\rfloor-\frac{1}{2}, & \text{if $x \notin \Z$}, \\
0, & \text{if $x \in \Z$},  \end{cases}$$
\end{defn}

The next proposition also follows from \cite[Theorem 2.8]{BL}.

\begin{prop}\label{LensWalker}
 For a lens space $L(p,q)$, $\lambda(L(p,q))=-\frac12s(q,p)$.
\end{prop}

When $p,q >0$, write $p/q$ as a continued fraction $$\frac{p}{q}=[a_1,\cdots, a_n]=a_1-\frac{1}{a_2-\displaystyle\frac{1}{
a_3-\displaystyle\frac{1}{\ddots}
}}.$$
We learn from  Rasmussen \cite[Lemma 4.3]{RasLens} that the Casson--Walker invariant of $L(p,q)$ can be calculated alternatively by the formula
\begin{equation}\label{eq:DedekindSum}
s(q,p)=\frac{1}{12}(\frac{q}{p}+\frac{q'}{p}+\sum_{i=1}^n(a_i-3))
\end{equation}
where $0<q'<p$ is the unique integer such that $qq'\equiv 1 \pmod p$.

\begin{lem}\label{lem:CWzero}
The Casson--Walker invariant of a lens space $\lambda(L(p,q))$
vanishes if and only if $q^2\equiv -1 \pmod p$.
\end{lem}
\begin{proof}
If $\lambda(L(p,q))=0$, we must have $q+q' \equiv 0 \pmod p$ in view of formula (\ref{eq:DedekindSum}).  Together with the definition of $q'$, we immediately see $$q^2 \equiv -qq'\equiv -1 \pmod p.$$

On the other hand, it is well known from the classification result of lens spaces that $L(p,q)$ is orientation-preserving homeomorphic to $L(p,q')$.  Hence, $\lambda(L(p,q))=\lambda(L(p,q'))$.  If $q^2\equiv -1\pmod p$, then $q'\equiv -q \pmod p$; so we have $\lambda(L(p,q))=\lambda(L(p,-q))=-\lambda(L(p,q))$.  This implies $\lambda(L(p,q))=0.$
\end{proof}

The Casson--Walker invariant is closely related to the correction
terms and the Euler characteristic of $HF_{\mathrm{red}}$. The
following theorem is established as \cite[Theorem 3.3]{Rustamov},
whose special case was also known in \cite[Theorem~5.1]{OSzAbGr}.

\begin{thm}\label{thm:Rust}
For a rational homology sphere $Y$, we have
$$ |H_1(Y;\Z)|\lambda(Y)=\sum_{\s\in \Spinc(Y)}(\chi(HF_{\mathrm{red}}(Y,\s))-\frac{1}{2}d(Y,\s)).$$
\end{thm}

\subsection{Casson--Gordon invariant}

 Let us recall the following $G$--signature theorem for closed four-manifolds \cite[Proposition~6.18]{AS}.

\begin{thm}[$G$--signature Theorem]
Suppose $ \tilde{X} \xrightarrow{\pi} X$ is an $m$--fold cyclic
cover of closed four-manifolds branched over a closed surface $F$
in $X$.  Then,
$$\operatorname{sig}(\tilde{X})=m\cdot\operatorname{sig} (X)-[F]^2\cdot
\frac{m^2-1}{3m}.$$ Here
$[F]^2=\langle\mathrm{PD}^{-1}([F])\smile\mathrm{PD}^{-1}([F]),[X]\rangle$.
\end{thm}

Consider a closed oriented three-manifold $Y$ with $H_1(Y;\Z)=\Z_m$.  It has a unique $m$--fold cyclic cover $\tilde{Y} \rightarrow Y$.  Pick up an $m$--fold cyclic branched covering of four-manifold $\tilde{W} \rightarrow W$, branched over a properly embedded surface $F$ in $W$, such that $\partial(\tilde{W}\rightarrow W)=(\tilde{Y}\rightarrow Y)$.  The existence of such $(W,F)$ follows from \cite[Lemma 2.2]{CG}.

\begin{defn}
The total Casson--Gordon invariant of $Y$ is given by $$\tau(Y) = m \cdot \text{sig}(W)-\text{sig}(\tilde{W})-[F]^2\cdot\frac{m^2-1}{3m}. $$
\end{defn}

It is a standard argument to see the independence of the definition on the choice of the four-manifolds cover $\tilde{W} \rightarrow W$.  Suppose $\tilde{W'} \rightarrow W'$ is another cover that bounds $\tilde{Y} \rightarrow Y$, then we can construct a branched cover $-\tilde{W'}\cup_{\tilde{Y}} \tilde{W} \rightarrow -W' \cup_Y W$ of closed four-manifolds.  It follows readily from Novikov additivity and the $G$--signature Theorem that the invariant is well defined.

\begin{defn}
Let $K$ be a knot in an integral homology sphere $Y$.  The {\it generalized signature function} $\sigma_K(\xi)$ is the signature of the matrix $A(\xi):=(1-\bar{\xi})A+(1-\xi)A^{T}$ for a Seifert matrix $A$ of $K$, where $|\xi|=1$.
\end{defn}

A surgery formula for the total Casson--Gordon invariant was established in \cite[Lemma 2.22]{BL}.

\begin{prop}
Let $K$ be a knot in an integral homology sphere $Y$, then
\begin{equation}\label{eq:surgery2}
 \tau(Y_{p/q}(K))=\tau(L(p,q))-\sigma(K,p).
\end{equation}
where $\sigma(K,p)=\sum_{r=1}^{p-1}\sigma_K(e^{2i\pi r/p})$.
\end{prop}

Quite amazingly, the total Casson--Gordon invariant of the lens
space $L(p,q)$ is also related to the Dedekind sum.
\begin{prop}\label{LensGordon}
For a lens space $L(p,q)$, $\tau(L(p,q))=-4p\cdot s(q,p)$.
\end{prop}

\subsection{Cosmetic surgeries with slopes of opposite signs}

In this subsection, we derive an obstruction for purely cosmetic
surgeries with slopes of opposite signs. Recall that both
$\mathrm{Spin}^c(S^3_{p/q}(K))$ and $\mathrm{Spin}^c(L(p,q))$ are
identified with $\mathbb Z/p\mathbb Z$. This leads to an explicit
identification of $\mathrm{Spin}^c(S^3_{p/q}(K))$ with
$\mathrm{Spin}^c(L(p,q))$ in the statement of the next
proposition.

\begin{prop}\label{prop:obst}
Given $p,q_1,q_2>0$ and a knot $K$ in $S^3$.  If
$Z=S^3_{p/q_1}(K)\cong S^3_{-p/q_2}(K)$ as oriented manifolds,
then
$$\Delta''_K(1)=0,$$
$$\sum_{\s\in \Spinc(Z)}\chi(HF_{\mathrm{red}}(Z,\s))=0,$$ and there exists a one-to-one correspondence  $$\sigma: \Spinc(L(p,q_1))\rightarrow\Spinc(L(p,q_2))$$ such that $$d(S^3_{p/q_1}(K),\s)=d(L(p,q_1),\s)=d(S^3_{-p/q_2}(K),\sigma(\s))
=-d(L(p,q_2),\sigma(\s))$$ for every $\s$.
\end{prop}

\begin{proof}
Using the surgery formulae (\ref{eq:surgery1}) (\ref{eq:surgery2}), we can compute the Casson--Walker and Casson--Gordon invariants of $Z$ from its two surgery presentations and obtain
$$\lambda(Z)=\lambda(L(p,q_1))+\frac{q_1}{2p}\Delta_K''(1)=
    \lambda(L(p,-q_2))+\frac{-q_2}{2p}\Delta_K''(1),$$
$$\tau(Z)=\tau(L(p,q_1))-\sigma(K,p)=\tau(L(p,-q_2))-\sigma(K,p).$$
  In light of Proposition \ref{LensWalker} and \ref{LensGordon}, we must have $\Delta''_K(1) = 0$ hence $$\lambda(Z)=\lambda(S^3_{p/q_1}(K))=\lambda(L(p,q_1)).$$
This, according to Theorem~\ref{thm:Rust}, implies
$$\sum_{\s\in \Spinc(Z)}(\chi(HF_{\mathrm{red}}(Z,\s)-\frac{1}{2}d(Z,\s))=\sum_{\s \in \Spinc(L(p,q_1))}-\frac{1}{2}d(L(p,q_1),\s).$$
It follows from Theorem~\ref{thm:Grading} that $$d(S^3_{p/q_1}(K),\s)\leq d(L(p,q_1),\s)$$ for any knot $K$ and $p/q_1>0$.  Therefore, $$\sum_{\s\in \Spinc(Z)}\chi(HF_{\mathrm{red}}(Z,\s))\leq 0.$$

On the other hand, $$\lambda(Z)=\lambda(S^3_{-p/q_2}(K))=\lambda(L(p,-q_2)).$$
Again, this implies
$$\sum_{\s\in \Spinc(Z)}(\chi(HF_{\mathrm{red}}(Z,\s)-\frac{1}{2}d(Z,\s))=\sum_{\s \in \Spinc(L(p,-q_2))}-\frac{1}{2}d(L(p,-q_2),\s).$$
With negative surgery coefficient $-p/q_2$,
Theorem~\ref{thm:Grading} implies that
$$d(S^3_{-p/q_2}(K),\s)=-d(S^3_{p/q_2}(\overline
K),\s)\geq-d(L(p,q_2),\s)= d(L(p,-q_2),\s).$$  Therefore,
$$\sum_{\s\in \Spinc(Z)}\chi(HF_{\mathrm{red}}(Z,\s))\geq 0.$$
This implies $$\sum_{\s\in
\Spinc(Z)}\chi(HF_{\mathrm{red}}(Z,\s))= 0$$ and
$$d(S^3_{p/q_i}(K),\s)= d(L(p,q_i),\s)$$ for $i=1,2$ and every
$\s$.
\end{proof}

It is a natural question to ask what three-manifolds may be
obtained via purely cosmetic surgeries on knots in $S^3$.  The
above obstruction enables us to eliminate the following class
 of three-manifolds that includes all Seifert fibred rational homology spheres.

\begin{cor}
 If $Z$ is a plumbed three-manifold of a negative-definite graph with at most one bad point, then $Z$ can not be obtained via purely cosmetic surgeries on knots in $S^3$.
\end{cor}
\begin{proof}
 By \cite[Corollary 1.4]{OSzPlumb}, all elements of $HF^+(Z)$ have even $\Z/2\Z$ grading.  This implies that in the case $HF_{\mathrm{red}}(Z)\neq 0$, it must be that
$$\sum_{\s\in \Spinc(Z)}\chi(HF_{\mathrm{red}}(Z,\s))=\mathrm{rank}\:HF_{\mathrm{red}}(Z)\neq 0,$$
hence we can apply Proposition~\ref{prop:obst}.
 The other case where $HF_{\mathrm{red}}(Z)=0$ follows from discussions in \cite{Wu}.
\end{proof}


\section{Proof of the main theorem}\label{sect:Main}
Let
$$\widehat{A}_k=C\{ \max\{ i,j-k\}=0\},\quad\widehat{B}=C\{ i=0\}$$
and
$$\nu(K)=\min\left\{k\in\mathbb
Z\left|\widehat{v}_k\co\widehat{A}_k\to\widehat{CF}(S^3) \text{
\small induces a non-trivial map in homology}\right.\right\}$$ be
the knot invariant defined by Ozsv\'ath--Szab\'o \cite{OSzRatSurg}.

\begin{prop}\label{prop:nu<=0}
Suppose that $K\subset S^3$ is a knot with $V_0=H_0=0$. Then
$\nu(K)\le0$.
\end{prop}

\begin{proof}
Consider the commutative diagram
$$
\begin{CD}
\widehat A_k&@>i_A>>&A^+_k\\
@V\widehat v_kVV&&@Vv^+_kVV\\
\widehat B&@>i_B>>&B^+
\end{CD}
$$
and the induced commutative diagram of homology.
Since $U\mathbf 1=0$, $\mathbf 1\in\mathfrak  A^+_k$ is in the image of $(i_A)_*$.
Since $V_0=0$, $\mathfrak v^+_0(\mathbf 1)=\mathbf 1$. The above commutative diagram shows that the induced map $(\widehat v_0)_*$ is nontrivial in homology. Thus $\nu(K)\le0$.
\end{proof}

\begin{proof}[Proof of Theorem~\ref{thm:OpSlope}]
By the result in \cite{Wu}, we only need to consider the case that $r_1,r_2$ have opposite signs. Suppose $r_1=p/q_1$ and $r_2=-p/q_2$, where $p,q_1,q_2$ are positive integers, $\gcd(p,q_1)=\gcd(p,q_2)=1$.
By Proposition~\ref{prop:obst}, $d(S^3_{p/q_1}(K),i)=d(L(p,q_1),i)$.
Theorem~\ref{thm:Grading} implies that $V_0=H_0=0$.
By Proposition~\ref{prop:nu<=0}, $\nu(K)\le0$. Since $\nu(K)=\tau(K) \text{ or }\tau(K)+1$ (see \cite[Lemma~9.2]{OSzRatSurg} and \cite{OSz4Genus,RasThesis}), $\tau(K)\le0$. The same argument can be applied to $\overline{K}$ to show that $\tau(\overline K)\le0$.
Since $\tau(\overline K)=-\tau(K)$, we must have $\tau(K)=0$.

Since $\nu(K)=\tau(K) \text{ or }\tau(K)+1$ and $\nu(K)\le0$, we must have $\nu(K)=0$. The same argument can be applied to $\overline K$ to show that $\nu(\overline K)=0$. So we can apply \cite[Proposition~9.9]{OSzRatSurg} to conclude that $r_1=-r_2$.

Using Proposition~\ref{prop:obst} and (\ref{eq:surgery1}), we conclude that $$\lambda(L(p,q_1))=\lambda(S^3_{p/q_1}(K))=\lambda(S^3_{-p/q_1}(K))=\lambda(L(p,-q_1))=-\lambda(L(p,q_1)).$$
So $\lambda(L(p,q_1))=0$. The fact that $q_1^2\equiv-1\pmod p$ follows from Lemma~\ref{lem:CWzero}.
\end{proof}


\section{The computation of $HF_{\mathrm{red}}$}\label{sect:Red}

In order to get more information about the knot $K$, we need to consider the reduced Heegaard Floer homology $HF_{\mathrm{red}}$ of the surgered manifolds. If $K$ admits purely cosmetic surgeries, our computation (Proposition~\ref{prop:Const}) shows that $HF_{\mathrm{red}}(S^3_{p/q}(K))$ looks like the Heegaard Floer homology of the surgery on an amphicheiral knot. Thus it provides more evdidence to Conjecture~\ref{conj:Cosmetic}
for knots in $S^3$.

Let $$\mathfrak A_{k,\mathrm{red}}=\mathfrak A_k^+/\mathfrak A^T_k$$ and
$$\mathcal A_{i,\mathrm{red}}=\bigoplus_{s\in\mathbb Z}(s,\mathfrak A_{\lfloor\frac{i+ps}q\rfloor,\mathrm{red}}(K)),$$
we have:

\begin{prop}\label{prop:RedIsom}
Suppose $K\subset S^3$ is a knot with $V_0=H_0=0$.  If either
$$p/q>0$$ or $$p/q<0,\quad d(S^3_{p/q}(K),i)=d(L(p,q),i),$$ then
$$\mathcal A_{i,\mathrm{red}}\cong HF_{\mathrm{red}}(S^3_{p/q}(K),i)$$ and the isomorphism preserves the absolute grading.
\end{prop}

\begin{lem}\label{lem:NoTors}
Suppose $K\subset S^3$ is a knot with $V_0=H_0=0$. If $p/q>0$, then $\mathfrak D^T_{i,p/q}$ is surjective and its kernel is isomorphic to $\mathcal T^+$;
if $p/q<0$, then $\mathfrak D^T_{i,p/q}$ is injective and its cokernel is isomorphic to $\mathcal T^+$.
\end{lem}
\begin{proof}
We always suppose $p>0$ and $0\le i\le p-1$.
First consider the case that $p/q>0$. The surjectivity of $\mathfrak D^T_{i,p/q}$ is guaranteed by Lemma~\ref{lem:>0Surj}.
We define a map $$\sigma\co \mathcal T^+\to H_*(\mathbb A^+_i)$$
as follows. Given $\xi\in \mathcal T^+$, let $\sigma(\xi)=\{(s,\xi_s)\}_{s\in\mathbb Z}$, where
$$
\xi_s=\left\{
\begin{array}{ll}
\xi,&\text{if }s=0\text{ or }-1,\\
U^{H_{\lfloor\frac{i+p(s-1)}q\rfloor}-V_{\lfloor\frac{i+ps}q\rfloor}}\xi_{s-1}=U^{H_{\lfloor\frac{i+p(s-1)}q\rfloor}}\xi_{s-1},&\text{if }s>0,\\
U^{V_{\lfloor\frac{i+p(s+1)}q\rfloor}-H_{\lfloor\frac{i+ps}q\rfloor}}\xi_{s+1}=U^{V_{\lfloor\frac{i+p(s+1)}q\rfloor}}\xi_{s+1},&\text{if }s<-1.
\end{array}
\right.
$$
We claim that there is a short exact sequence $$
\begin{CD}
0@>>> \mathcal T^+@>\sigma>> \mathcal A_i^T @>\mathfrak D^T_{i,p/q}>> H_*(\mathbb B_i^+)@>>>0.
\end{CD}
$$
In fact, $\sigma$ is clearly injective and the image of $\sigma$ is in the kernel of $\mathfrak D^T_{i,p/q}$.
Suppose $\{(s,\xi_s)\}_{s\in\mathbb Z}$ is in the kernel of $\mathfrak D^T_{i,p/q}$, we want to show that it is in the image of $\sigma$.
Since $V_{\lfloor\frac iq \rfloor}=H_{\lfloor\frac{i+p(-1)}q \rfloor}=0$, one must have $\xi_{-1}=\xi_0$. Let $\xi=\xi_0$, then we can check $\sigma(\xi)=\{(s,\xi_s)\}$.
This finishes the proof of the case where $p/q>0$.

Next consider the case where $p/q<0$.
We have \begin{equation}\label{eq:q<0VH}
V_{\lfloor\frac{i+ps}q \rfloor}=0 \text{ when } s<0,\quad H_{\lfloor\frac{i+ps}q \rfloor}=0\text{ when }s\ge0.
\end{equation}
Suppose
$$\{(s,\xi_s)\}_{s\in\mathbb Z}$$
is in the kernel of $\mathfrak D^T_{i,p/q}$. Then
\begin{equation}\label{eq:XiKer}
U^{H_{\lfloor\frac{i+p(s-1)}q\rfloor}}\xi_{s-1}+U^{V_{\lfloor\frac{i+ps}q\rfloor}}\xi_{s}=0,
\end{equation}
for any $s\in\mathbb Z$.
By the definition of direct sum, $\xi_s=0$ when $|s|$ is sufficiently large. Suppose $\xi_s=0$ for some $s>0$, then
it follows from (\ref{eq:q<0VH}) and (\ref{eq:XiKer}) that $\xi_{s-1}=0$. So we have $\xi_s=0$ for all $s\ge0$. Similarly, $\xi_s=0$ for all $s<0$.
This proves that $\mathfrak D^T_{i,p/q}$ is injective.

Given
$$\mbox{\boldmath${\eta}$}=\{(s,\eta_s)\}_{s\in\mathbb Z}\in H_*(\mathbb B^+),$$
let
$$\phi(\mbox{\boldmath${\eta}$})=\eta_0+\sum_{s>0}U^{\sum_{j=1}^sV_{\lfloor\frac{i+p(j-1)}q
\rfloor}}\eta_s+\sum_{s<0}U^{\sum_{j=1}^{-s}H_{\lfloor\frac{i-pj}q
\rfloor}}\eta_s.$$ We claim that the sequence
$$
\begin{CD}
0@>>>\mathcal A_i^T @>\mathfrak D^T_{i,p/q}>>H_*(\mathbb B_i^+)@>\phi>>\mathcal T^+@>>>0.
\end{CD}
$$
is exact.  It is routine to check $\phi\circ\mathfrak  D^T_{i,p/q}=0$.
Moreover, suppose $\phi(\mbox{\boldmath${\eta}$})=0$. Let $M>0$ be
an integer such that $\eta_s=0$ whenever $|s|> M$. Define
$$
\xi_s=\left\{
\begin{array}{ll}
\eta_s, &\text{if }s\le -M,\\
\eta_s-U^{H_{\lfloor \frac{i+p(s-1)}q\rfloor}}\xi_{s-1}, &\text{if }-M<s<0,\\
\eta_{s+1},&\text{if }s\ge M-1,\\
\eta_{s+1}-U^{V_{\lfloor \frac{i+p(s+1)}q\rfloor}}\xi_{s+1}, &\text{if }0\le s<M-1.
\end{array}
\right.
$$
We can check $\mathfrak D^T_{i,p/q}\{(s,\xi_s)\}=\mbox{\boldmath${\eta}$}$, where at $(0,\eta_0)$ we use the fact that $\phi(\mbox{\boldmath${\eta}$})=0$.
This proves $\ker\phi=\mathrm{im}\:\mathfrak D^T_{i,p/q}$.
The image of $\phi$ is clearly $\mathcal T^+$, so $\mathcal T^+$ is isomorphic to the cokernel of $\mathfrak D^T_{i,p/q}$.
\end{proof}

\begin{proof}[Proof of Propostion~\ref{prop:RedIsom}]
When $p/q>0$, we can identify $HF^+(S^3_{p/q},i)$ with the kernel of $\mathfrak D^+_{i,p/q}$. Then there is a natural projection map
$$\pi\co HF^+(S^3_{p/q},i)\to\mathcal A_{i,\mathrm{red}}$$
We claim that there is a short exact sequence
$$
\begin{CD}
0@>>> \mathcal T^+@>\sigma>>HF^+(S^3_{p/q},i)@>\pi>>\mathcal A_{i,\mathrm{red}}@>>>0,
\end{CD}
$$
where $\sigma$ is the map defined in Lemma~\ref{lem:NoTors}.

From Lemma~\ref{lem:NoTors} we know that $\sigma$ is injective, and $\mathrm{im}\:\sigma\subset\mathrm{ker}\:\pi$.
If $\mbox{\boldmath${\xi}$}\in\ker\mathfrak D^+_{i,p/q}$ is in the kernel of $\pi$, then $\mbox{\boldmath${\xi}$}$ is contained in $\mathcal A^T_{i,p/q}$, so
$$\mbox{\boldmath${\xi}$}\in\ker\mathfrak D^T_{i,p/q}=\mathrm{im}\:\sigma.$$

Next we show that $\pi$ is surjective. Let $\pi'\co\mathcal A_i^+\to\mathcal A_{i,\mathrm{red}}$ be the projection map. We need to show that for any $\mbox{\boldmath${\zeta}$}\in \mathcal A_{i,\mathrm{red}}$,
there exists a $\mbox{\boldmath${\xi}$}\in\ker\mathfrak D^+_{i,p/q}$ with $\pi'(\mbox{\boldmath${\xi}$})=\mbox{\boldmath${\zeta}$}$. In fact, let $\mbox{\boldmath${\xi}$}_1$ be any element with
$\pi'(\mbox{\boldmath${\xi}$}_1)=\mbox{\boldmath${\zeta}$}$. Since $\mathfrak D^T_{i,p/q}$ is surjective, there exists $\mbox{\boldmath${\xi}$}_2\in\mathcal A^T_i$ with $$\mathfrak D^T_{i,p/q}(\mbox{\boldmath${\xi}$}_2)=\mathfrak D^+_{i,p/q}(\mbox{\boldmath${\xi}$}_1),$$ then $\mbox{\boldmath${\xi}$}=\mbox{\boldmath${\xi}$}_1-\mbox{\boldmath${\xi}$}_2$ is the element we want. This finishes the proof of the claim.

The claim immediately implies our conclusion when $p/q>0$.

When $p/q<0$, suppose $d(S^3_{p/q}(K),i)=d(L(p,q),i)$. We claim that
$$
\mathrm{im}\:\mathfrak D^+_{i,p/q}=\mathrm{im}\:\mathfrak D^T_{i,p/q}.
$$
Otherwise, $\mathrm{im}\:\mathfrak D^+_{i,p/q}$ is strictly larger than $\mathrm{im}\:\mathfrak D^T_{i,p/q}$. Then $\phi(\mathrm{im}\:\mathfrak D^+_{i,p/q})$ would contain a nonzero element,
where $\phi$ is the map defined in Lemma~\ref{lem:NoTors}. Hence $\mathbf1\in\phi(\mathrm{im}\:\mathfrak D^+_{i,p/q})$. By the exact triangle (\ref{eq:TriangleMC}), 
$U^nHF^+(S^3_{p/q}(K),i)$ is contained in $\mathrm{incl}_*(\mathrm{coker}\:\mathfrak D^+_{i,p/q})$ when $n\gg0$.
It follows that the bottommost element in $U^nHF^+(S^3_{p/q}(K),i)$ for $n\gg0$ has grading higher than the grading of $(0,\mathbf1)\in(0,B^+)$, which is $d(L(p,q),i)$. This gives a contradiction.

Now our conclusion easily follows from the claim and Lemma~\ref{lem:NoTors}.
\end{proof}

By the proof of Theorem~\ref{thm:OpSlope},
Proposition~\ref{prop:Const} is an easy corollary of the following proposition.

\begin{prop}
Suppose $K\subset S^3$ is a knot with $V_0=H_0=0$. Then there exists a constant $C=C(K)$, such that
$$\mathrm{rank}\:HF_{\mathrm{red}}(S^3_{p/q}(K))=|q|\cdot C,$$
for any coprime integers $p,q$ with $p/q>0$. Moreover, if $d(S^3_{p/q}(K),i)=d(L(p,q),i)$ for all $i$, then the above equality also holds for $p/q<0$.
\end{prop}
\begin{proof}
Let $C=\sum_{k\in\mathbb Z}\mathrm{rank}\:\mathfrak A_{k,\mathrm{red}}$. In
$$\bigoplus_{i\in\mathbb Z/p\mathbb Z}\mathcal A_{i,\mathrm{red}}=\bigoplus_{i=0}^{p-1}\bigoplus_{s\in\mathbb Z}(s,\mathfrak A_{\lfloor\frac{i+ps}q\rfloor,\mathrm{red}}(K)),$$
each $\mathfrak A_{k,\mathrm{red}}$ appears exactly $|q|$ times.
It follows from Proposition~\ref{prop:RedIsom} that
$$\mathrm{rank}\:HF_{\mathrm{red}}(S^3_{p/q}(K))=|q|\cdot C,$$
whenever the conditions in the statement of the theorem are
satisfied. Note that the constant $C(K)$ is indeed the same as the
rank of $HF_{\mathrm{red}}(S^3_{p}(K))$ for $p>0$.
\end{proof}


\begin{thebibliography}{H}

\bibitem{AS}{\bf M. Atiyah, I. Singer}, {\it The index of elliptic operators. III}, Ann. of Math. (2) 87 (1968) 546--604.

\bibitem{BL}{\bf S. Boyer, D. Lines}, {\it Surgery formulae for Casson's invariant and extensions to homology lens space}, J. Reine Angew. Math. 405 (1990), 181--220.

\bibitem{CG}{\bf A. Casson, C. Gordon}, {\it On slice knots in dimension three}, Proc. Sympos. Pure Math., vol 32, Amr. Math.Soc., Providence, R.I., 39--53.

\bibitem{GordonICM}{\bf C. Gordon}, {\it Dehn surgery on knots}, Proceedings of the International Congress of Mathematicians, Vol. I, II (Kyoto, 1990), 631--642, Math. Soc. Japan, Tokyo, 1991.

\bibitem{GL}{\bf C. Gordon, J. Luecke}, {\it Knots are determined by their complements}, J. Amer. Math. Soc.  2  (1989),  no. 2, 371--415.

\bibitem{Kirby}{\bf R. Kirby}, {\it Problems in low-dimensional
topology}, Geometric topology, Proceedings of the 1993 Georgia
International Topology Conference held at the University of
Georgia, Athens, GA, August 2--13, 1993. Edited by William H.
Kazez. AMS/IP Studies in Advanced Mathematics, 2.2. American
Mathematical Society, Providence, RI; International Press,
Cambridge, MA, 1997.

\bibitem{Mathieu}{\bf Y. Mathieu}, {\it Closed $3$--manifolds unchanged by Dehn surgery},
J. Knot Theory Ramifications 1 (1992), no. 3, 279--296.

\bibitem{OSzAnn1}{\bf P. Ozsv\'ath, Z. Szab\'o}, {\it Holomorphic disks and topological invariants for closed
three-manifolds}, Ann. of Math. (2), 159 (2004), no. 3,
1027--1158.


\bibitem{OSzAbGr}{\bf P. Ozsv\'ath, Z. Szab\'o}, {\it Absolutely graded Floer homologies and intersection forms for four-manifolds with boundary}, Adv. Math.  173  (2003),  no. 2, 179--261.

\bibitem{OSzKnot}{\bf P. Ozsv\'ath, Z. Szab\'o}, {\it Holomorphic disks and knot invariants},
Adv. Math. 186 (2004), no. 1, 58--116.

\bibitem{OSzPlumb}{\bf P Ozsv\'ath, Z Szab\'o}, {\it On the Floer homology of plumbed three-manifolds},  Geom. Topol.  7  (2003), 185--224 (electronic).

\bibitem{OSz4Genus}{\bf P Ozsv\'ath, Z Szab\'o}, {\it Knot Floer homology and the four-ball genus},  Geom. Topol.  7  (2003), 615--639.




\bibitem{OSzIntSurg}{\bf P. Ozsv\'ath, Z. Szab\'o}, {\it Knot Floer homology and integer surgeries},  Algebr. Geom. Topol.  8  (2008),  no. 1, 101--153.

\bibitem{OSzRatSurg}{\bf P. Ozsv\'ath, Z. Szab\'o}, {\it Knot Floer homology and rational
surgeries}, Algebr. Geom. Topol. 11 (2011), no. 1, 1--68.

\bibitem{RasThesis}{\bf J. Rasmussen}, {\it Floer homology and knot complements}, PhD Thesis,
Harvard University (2003), available at arXiv:math.GT/0306378.

\bibitem{RasLens}{\bf J. Rasmussen}, {\it Lens space surgeries and a conjecture of Goda and Teragaito}, Geom. Topol. 8 (2004), 1013--1031.

\bibitem{Rustamov}{\bf R. Rustamov}, {\it Surgery formula for the renormalized Euler characteristic of Heegaard Floer homology}, preprint (2004), available at arXiv:math.GT/0409294.

\bibitem{Wang}{\bf J. Wang}, {\it Cosmetic surgeries on genus one knots},  Algebr. Geom. Topol.  6  (2006), 1491--1517.

\bibitem{Wu}{\bf Z. Wu}, {\it Cosmetic Surgery in Integral Homology $L$-Spaces}, Geom. Topol. 15 (2011), no. 2, 1157--1168.

\end{thebibliography}
\end{document}